\documentclass[11pt]{amsart}
\usepackage{amsmath,amsthm,amssymb,xypic,array}
\usepackage[T1]{fontenc}
\usepackage{amsfonts}
\usepackage{amsmath}
\usepackage{pgf,tikz}
\usepackage{mathrsfs}
\usetikzlibrary{arrows}
\usepackage{amssymb}
\usepackage{amsthm}
\usepackage{setspace}
\usepackage[cp850]{inputenc}
\usepackage{hyperref}
\usepackage{graphicx}
\usepackage{vmargin}
\usepackage{fancyhdr}
\usepackage{booktabs}
\usepackage{longtable}
\usepackage{geometry}
\usetikzlibrary{trees}

\setpapersize{A4}
\providecommand{\U}[1]{\protect\rule{.1in}{.1in}}
\providecommand{\U}[1]{\protect\rule{.1in}{.1in}}
\providecommand{\U}[1]{\protect\rule{.1in}{.1in}}
\providecommand{\U}[1]{\protect\rule{.1in}{.1in}}
\providecommand{\U}[1]{\protect\rule{.1in}{.1in}}
\input{xy}
\xyoption{all}
\theoremstyle{theorem}

\newtheorem{Theorem}{Theorem}[section]
\newtheorem{theoremn}{Theorem}

\newtheorem{Lemma}[Theorem]{Lemma}
\newtheorem{Proposition}[Theorem]{Proposition}

\theoremstyle{definition}
\newtheorem{Construction}[Theorem]{Construction}
\newtheorem{Definition}[Theorem]{Definition}
\newtheorem{Remark}[Theorem]{Remark}

\numberwithin{equation}{section}
\newcommand{\arXiv}[1]{\href{http://arxiv.org/abs/#1}{arXiv:#1}}
\newcommand{\Spec}{\operatorname{Spec}}
\newcommand{\logpol}{\operatorname{log}}
\newcommand{\Aut}{\operatorname{Aut}}

\newcommand{\ch}{\operatorname{char}}

\newcommand{\Hom}{\operatorname{Hom}}
\newcommand{\Ext}{\operatorname{Ext}}

\newcommand{\im}{\operatorname{Im}}

\newcommand{\codim}{\operatorname{codim}}
\newcommand{\Sing}{\operatorname{Sing}}
\newcommand{\mult}{\operatorname{mult}}

\def\bibaut#1{{\sc #1}}
\newcommand{\quot}{/\hspace{-1.2mm}/}

\begin{document}

\title{On the rigidity of moduli of weighted pointed stable curves}

\author[Barbara Fantechi]{Barbara Fantechi}
\address{\sc Barbara Fantechi\\
SISSA\\
via Bonomea 265\\
34136 Trieste\\ Italy}
\email{fantechi@sissa.it}

\author[Alex Massarenti]{Alex Massarenti}
\address{\sc Alex Massarenti\\
Universidade Federal Fluminense\\
Rua M\'ario Santos Braga\\
24020-140, Niter\'oi, Rio de Janeiro\\ Brazil}
\email{alexmassarenti@id.uff.br}

\date{\today}
\subjclass[2010]{Primary 14H10; Secondary 14D22, 14D23, 14D06}
\keywords{Moduli of weighted curves curves, infinitesimal deformations, positive characteristic, automorphisms}

\maketitle

\begin{abstract}
Let $\overline{\mathcal{M}}_{g,A[n]}$ be the Hassett moduli stack of weighted stable curves, and let $\overline{M}_{g,A[n]}$ be its coarse moduli space. These are compactifications of $\mathcal{M}_{g,n}$ and $M_{g,n}$ respectively, obtained by assigning rational weights $A = (a_{1},...,a_{n})$, $0< a_{i} \leq 1$ to the markings; they are defined over $\mathbb{Z}$, and therefore over any field. We study the first order infinitesimal deformations of $\overline{\mathcal{M}}_{g,A[n]}$ and $\overline{M}_{g,A[n]}$. In particular, we show that $\overline{M}_{0,A[n]}$ is rigid over any field, if $g\geq 1$ then $\overline{\mathcal{M}}_{g,A[n]}$ is rigid over any field of characteristic zero, and if $g+n > 4$ then the coarse moduli space $\overline{M}_{g,A[n]}$ is rigid over an algebraically closed field of characteristic zero. Finally, we take into account a degeneration of Hassett spaces parametrizing rational curves obtained by allowing the weights to have sum equal to two. In particular, we consider such a Hassett $3$-fold which is isomorphic to the Segre cubic hypersurface in $\mathbb{P}^4$, and we prove that its family of first order infinitesimal deformations is non-singular of dimension ten, and the general deformation is smooth.
\end{abstract}

\setcounter{tocdepth}{1}
\tableofcontents

\section*{Introduction}
In \cite{Ha} B. Hassett introduced new compactifications $\overline{\mathcal{M}}_{g,A[n]}$ of the moduli stack $\mathcal{M}_{g,n}$  parametrizing smooth genus $g$ curves with $n$ marked points, where the notion of stability is defined in terms of a fixed vector of rational weights $A[n] = (a_{1},...,a_{n})$, on the markings. The classical Deligne-Mumford compactification corresponds to the weights $a_1 = ... = a_n = 1$; Hassett construction requires that $0<a_i\le 1$ for every $i$ and that $\sum a_i>2-2g$.

As the stack $\overline{\mathcal{M}}_{g,n}$, the stacks $\overline{\mathcal{M}}_{g,A[n]}$ are smooth and proper over $\mathbb{Z}$, and therefore $\overline{\mathcal{M}}_{g,A[n]}^R$ is defined over any commutative ring $R$ via base change. By \cite{KM} the formation of the coarse moduli space is compatible with flat base change; we write $\overline{M}_{g,n}^R$ for the coarse moduli scheme of $\overline{\mathcal{M}}_{g,n}^R$, and refer to it as a Hassett moduli space. Again in analogy with the Deligne-Mumford case, Hassett stacks for $g=0$ are already schemes, hence coincide with the corresponding Hassett spaces.

Hassett spaces are central objects in the study of the birational geometry of $\overline{M}_{g,n}$. Indeed, in genus zero some of these spaces appear as intermediate steps of the blow-up construction of $\overline{M}_{0,n}$ developed by M. Kapranov in \cite{Ka} and some of them turn out to be Mori Dream Spaces \cite[Section 6]{AM}, while in higher genus they may be related to the LMMP on $\overline{M}_{g,n}$ \cite{Moo}.

In this paper we push forward the techniques developed in \cite{FM} to study the infinitesimal deformations of Hassett moduli stacks and spaces over an arbitrary field. The results in Theorems \ref{righas} and \ref{rigc} can be summarized as follows.

\begin{theoremn}
Let $K$ be an arbitrary field, and $n\ge 3$ an integer. Then the genus zero Hassett moduli space $\overline{M}_{0,A[n]}^K$ is rigid for any vector of weights $A[n]$.

Let $g\geq 1$ and assume $K$ is a field of characteristic zero. Then Hassett stack $\overline{\mathcal{M}}_{g,A[n]}^K$ is rigid for any vector of weights $A[n]$.
\end{theoremn}

For a field $K$ of characteristic zero we then apply the deformation theory of varieties with transversal $A_1$ and $\frac{1}{3}(1,1)$ singularities developed in \cite[Sections 5.4, 5.5]{FM} to the study of infinitesimal deformations of the coarse moduli spaces $\overline{M}_{g,A[n]}^K$. In the following statement we summarize the result on deformations of $\overline{M}_{g,A[n]}^K$ in Proposition \ref{ltdhas} and Theorem \ref{rigc}.

\begin{theoremn}
Let $K$ be a field of characteristic zero. If $g+n\geq 4$ then the coarse moduli space $\overline{M}_{g,A[n]}^K$ does not have locally trivial first order infinitesimal deformations for any vector of weights $A[n]$.

If $K$ is an algebraically closed field of characteristic zero and $g+n >4$ then $\overline{M}_{g,A[n]}^K$ is rigid for any vector of weights $A[n]$. 
\end{theoremn}

In Section \ref{Ssegre} we consider a natural variation $\overline{M}_{0,\widetilde{A}[n]}$ on the moduli problem of weighted pointed rational curves, introduced by B. Hassett in \cite[Section 2.1.2]{Ha} by allowing the weights to have sum equal to two.

In particular, we consider Hassett space $\overline{M}_{0,\widetilde{A}[6]}$ with weights $a_1=...=a_6 = \frac{1}{3}$. This space is isomorphic to the Segre cubic, a $3$-fold of degree three in $\mathbb{P}^4$ with ten nodes which carries a very rich projective geometry \cite{Do15}. In Section \ref{Ssegre} we study the infinitesimal deformations of $\overline{M}_{0,\widetilde{A}[6]}$, that is of the Segre cubic.

In Theorem \ref{segre} we prove that $\overline{M}_{0,\widetilde{A}[6]}$ does not have locally trivial deformations, while its family of first order infinitesimal deformations is non-singular of dimension ten and the general deformation is smooth.

Finally, in Section \ref{Saut} we apply the rigidity results in Section \ref{righassett}, and the techniques developed in \cite[Section 1]{FM} to lift automorphisms from zero to positive characteristic, in order to extend the main results on the automorphism groups of Hassett spaces in \cite{MM14}, \cite{MM}, \cite{BM}, \cite{Ma} and \cite{Ma16} over an arbitrary field. 

\subsection*{Acknowledgments}
The authors are members of the Gruppo Nazionale per le Strutture Algebriche, Geometriche e le loro Applicazioni of the Istituto Nazionale di Alta Matematica "F. Severi" (GNSAGA-INDAM).

\section{Preliminaries on Hassett moduli spaces}\label{HK}
Let $S$ be a noetherian scheme and $g,n$ two non-negative integers. A family of nodal curves of genus $g$ with $n$ marked points over $S$ consists of a flat proper morphism $\pi:C\rightarrow S$ whose geometric fibers are nodal connected curves of arithmetic genus $g$, and sections $s_{1},...,s_{n}$ of $\pi$. A collection of input data $(g,A[n]) := (g, a_{1},...,a_{n})$ consists of an integer $g\geq 0$ and the weight data: an element $(a_{1},...,a_{n})\in\mathbb{Q}^{n}$ such that $0<a_{i}\leq 1$ for $i = 1,...,n$, and
$$2g-2 + \sum_{i = 1}^{n}a_{i} > 0.$$
The vector $A[n]$ in the input data $(g,A[n])$ is called an admissible weight data.
\begin{Definition}\label{defha}
A family of nodal curves with marked points $\pi:(C,s_{1},...,s_{n})\rightarrow S$ is stable of type $(g,A[n])$ if
\begin{itemize}
\item[-] the sections $s_{1},...,s_{n}$ lie in the smooth locus of $\pi$, and for any subset $\{s_{i_{1}},..., s_{i_{r}}\}$ with non-empty intersection we have $a_{i_{1}} +...+ a_{i_{r}} \leq 1$,
\item[-] $\omega_{\pi}(\sum_{i=1}^{n}a_{i}s_{i})$ is $\pi$-relatively ample, where $\omega_{\pi}$ is the relative dualizing sheaf.
\end{itemize}
\end{Definition}
By \cite[Theorem 2.1]{Ha} given a collection $(g,A[n])$ of input data, there exists a connected Deligne-Mumford stack $\overline{\mathcal{M}}_{g,A[n]}$, smooth and proper over $\mathbb{Z}$, representing the moduli problem of pointed stable curves of type $(g,A[n])$. The corresponding coarse moduli scheme $\overline{M}_{g,A[n]}$ is projective over $\mathbb{Z}$.

Furthermore, by \cite[Proposition 3.8]{Ha} a weighted pointed stable curve admits no infinitesimal automorphisms, and its infinitesimal deformation space is unobstructed of dimension $3g-3+n$. Then $\overline{\mathcal{M}}_{g,A[n]}$ is a smooth Deligne-Mumford stack of dimension $3g-3+n$.   

For fixed $g,n$, consider two collections of weight data $A[n],B[n]$ such that $a_i\geq b_i$ for any $i = 1,...,n$. Then there exists a birational \textit{reduction morphism}
$$\rho_{B[n],A[n]}:\overline{M}_{g,A[n]}\rightarrow\overline{M}_{g,B[n]}$$
associating to a curve $[C,s_1,...,s_n]\in\overline{M}_{g,A[n]}$ the curve $\rho_{B[n],A[n]}([C,s_1,...,s_n])$ obtained by collapsing components of $C$ along which $\omega_C(b_1s_1+...+b_ns_n)$ fails to be ample, where $\omega_C$ denotes the dualizing sheaf of $C$.\\
Along the paper, when no confusion arises, we will denote reduction morphisms simply by $\rho$ omitting the weight data.

\begin{Remark}\label{n<3}
If $n\leq 2$ the reduction morphism $\rho:\overline{M}_{g,n}\rightarrow\overline{M}_{g,A[n]}$ contracts at most rational tails with two marked points, such rational tails do not have moduli. Therefore $\rho$ is an isomorphism and $\overline{M}_{g,A[n]}\cong \overline{M}_{g,n}$, see also \cite[Corollary 4.7]{Ha}.
\end{Remark}

The boundary of $\overline{M}_{g,A[n]}$, as for $\overline{M}_{g,n}$, has a stratification whose loci, called strata, parametrize curves of a fixed topological type and with a fixed configuration of the marked points.\\
We denote by $D_{I,J}$ the divisor parametrizing curves with two smooth components, of genus zero and $g$ respectively, intersecting in one node, where the points indexed by $I$ and $J$ lie on the genus zero and on the genus $g$ component respectively. 

Note that in $\overline{M}_{g,A[n]}$ may appear boundary divisors parametrizing smooth curves. For instance, as soon as there exist two indices $i,j$ such that $a_i+a_j\leq 1$ we get a boundary divisor whose general point represents a smooth curve where the marked points labeled by $i$ and $j$ collide. 

Finally we recall the notion of $\psi$-classes on $\overline{\mathcal{M}}_{g,A[n]}$. Let $(\pi:\mathcal{C}\rightarrow \overline{\mathcal{M}}_{g,A[n]}, (s_1,...,s_n))$ be the universal family on $\overline{\mathcal{M}}_{g,A[n]}$. The \textit{$\psi$-classes} on $\overline{\mathcal{M}}_{g,A[n]}$ are defined as
$$\psi_i = \pi_{*}(-s_i^{2}) = c_1(\sigma_i^{*}\omega_{\pi})$$
for $i = 1,...,n$.
\subsection{Kapranov's blow-up construction}
In \cite{Ka} M. Kapranov works, for sake of simplicity, on an algebraically closed field of characteristic zero. On the other hand Kapranov's arguments are purely algebraic and his description works over $\mathbb{Z}$.

In \cite{Ka} Kapranov proved that $\overline{M}_{0,n}$ can be constructed as an iterated blow-up $f_n:\overline{M}_{0,n}\rightarrow\mathbb{P}^{n-3}$ induced by $\psi_n$ which is big and globally generated. 
\begin{Construction}\cite{Ka}\label{kblu}
More precisely, fix $(n-1)$-points $p_{1},...,p_{n-1}\in\mathbb{P}^{n-3}$ in linear general position.
\begin{itemize}
\item[(1)] Blow-up the points $p_{1},...,p_{n-2}$, the strict transforms of the lines spanned by two of these $n-2$ points,..., the strict transforms of the linear spaces spanned by the subsets of cardinality $n-4$ of $\{p_{1},...,p_{n-2}\}$.  
\item[(2)] Blow-up $p_{n-1}$, the strict transforms of the lines spanned by pairs of points including $p_{n-1}$ but not $p_{n-2}$,..., the strict transforms of the linear spaces spanned by the subsets of cardinality $(n-4)$ of $\{p_{1},...,p_{n-1}\}$ containing $p_{n-1}$ but not $p_{n-2}$.\\
\vdots
\item[($r$)] Blow-up the strict transforms of the linear spaces spanned by subsets of the form 
$$\{p_{n-1},p_{n-2},...,p_{n-r+1}\}$$ 
so that the order of the blow-ups in compatible by the partial order on the subsets given by inclusion.\\ 
\vdots
\item[($n-3$)] Blow-up the strict transforms of the codimension two linear space spanned by the subset $\{p_{n-1},p_{n-2},...,p_{4}\}$.
\end{itemize}
The composition of these blow-ups is the morphism $f_{n}:\overline{M}_{0,n}\rightarrow\mathbb{P}^{n-3}$ induced by the psi-class $\psi_{n}$.

We denote by $W_{r,s}[n]$, where $s = 1,...,n-r-2$, the variety obtained at the $r$-th step once we finish blowing-up the subspaces spanned by subsets $S$ with $|S|\leq s+r-2$, and by $W_{r}[n]$ the variety produced at the $r$-th step. In particular, $W_{1,1}[n] = \mathbb{P}^{n-3}$ and $W_{n-3}[n] = \overline{M}_{0,n}$.
\end{Construction}
In \cite[Section 6.1]{Ha}, Hassett interprets the intermediate steps of Construction \ref{kblu} as moduli spaces of weighted rational curves. Consider the weight data 
$$A_{r,s}[n]:= (\underbrace{1/(n-r-1),...,1/(n-r-1)}_{(n-r-1)-\rm{times}}, s/(n-r-1), \underbrace{1,...,1}_{r-\rm{times}})$$ 
for $r = 1,...,n-3$ and $s = 1,...,n-r-2$. Then $W_{r,s}[n]\cong\overline{M}_{0,A_{r,s}[n]}$, and the Kapranov's map $f_{n}:\overline{M}_{0,n}\rightarrow\mathbb{P}^{n-3}$ factorizes as a composition of reduction morphisms
$$
\begin{array}{l}
\rho_{A_{r,s-1}[n],A_{r,s}[n]}:\overline{M}_{0,A_{r,s}[n]}\rightarrow\overline{M}_{0,A_{r,s-1}[n]},\: s = 2,...,n-r-2,\\ 
\rho_{A_{r,n-r-2}[n],A_{r+1,1}[n]}:\overline{M}_{0,A_{r+1,1}[n]}\rightarrow\overline{M}_{0,A_{r,n-r-2}[n]}.
\end{array} 
$$
\begin{Remark}\label{LM}
Hassett space $\overline{M}_{0,A_{1,n-3}[n]}$, that is $\mathbb{P}^{n-3}$ blown-up at all the linear spaces of codimension at least two spanned by subsets of $n-2$ points in linear general position, is the Losev-Manin's moduli space $\overline{L}_{n-2}$ introduced by A. Losev and Y. Manin in \cite{LM}, see \cite[Section 6.4]{Ha}. The space $\overline{L}_{n-2}$ parametrizes $(n-2)$-pointed chains of projective lines $(C,x_{0},x_{\infty},x_{1},...,x_{n-2})$ where:
\begin{itemize}
\item[-] $C$ is a chain of smooth rational curves with two fixed points $x_{0},x_{\infty}$ on the extremal components,
\item[-] $x_{1},...,x_{n-2}$ are smooth marked points different from $x_{0},x_{\infty}$ but non necessarily distinct,
\item[-] there is at least one marked point on each component.
\end{itemize}
By \cite[Theorem 2.2]{LM} there exists a smooth, separated, irreducible, proper scheme representing this moduli problem. Note that after the choice of two marked points in $\overline{M}_{0,n}$ playing the role of $x_{0},x_{\infty}$ we get a birational morphism $\overline{M}_{0,n}\rightarrow\overline{L}_{n-2}$ which is nothing but a reduction morphism.

For example, $\overline{L}_{1}$ is a point parametrizing a $\mathbb{P}^{1}$ with two fixed points and a free point, $\overline{L}_{2}\cong\mathbb{P}^{1}$, and $\overline{L}_{3}$ is $\mathbb{P}^{2}$ blown-up at three points in general position, that is a del Pezzo surface of degree six, see \cite[Section 6.4]{Ha} for further generalizations.
\end{Remark}

\subsection{Some notions of deformation theory}
Let us recall some basic notions of deformation theory to which we will constantly refer along the paper.

Let $X$ be a scheme over a field $K$, $A$ an Artinian $K$-algebra with residue field $K$. A deformation $X_A$ of $X$ over $\Spec(A)$ is called {\em trivial} if it is isomorphic to $X\times_K\Spec(A)$; it is {\em locally trivial} if there is an open cover of $X$ by open affines $U$ such that the induced deformation $U_A$ is trivial.

We recall some well-known facts about infinitesimal deformations of normal varieties. By \cite{Ill} the tangent and obstruction spaces to deformations of $X$ are given by $\Ext^1(L_X,\mathcal{O}_X)$ and $\Ext^2(L_X,\mathcal O_X)$ where $L_X$ is the cotangent complex; when $X$ is a normal variety, these spaces are actually $\Ext^1(\Omega_X,\mathcal{O}_X)$ and $\Ext^2(\Omega_X,\mathcal{O}_X)$ respectively. Locally trivial infinitesimal deformations have as tangent and obstruction spaces $H^1(X,T_X)$ and $H^2(X,T_X)$, respectively. 

\begin{Definition} Let $X$ be a scheme over a field. We will say it is {\em rigid} if it has no non-trivial infinitesimal deformations. If $X$ is smooth, this is equivalent to $H^1(X,T_X)=0$, and if $X$ is generically reduced this is equivalent to $\Ext^1(\Omega_X,\mathcal O_X)=0$.
\end{Definition}

By the exact sequence 
$$0\mapsto H^{1}(X,T_{X})\rightarrow\Ext^{1}(\Omega_{X},\mathcal{O}_{X})\rightarrow H^{0}(X,\mathcal{E}xt^{1}(\Omega_{X},\mathcal{O}_{X}))\rightarrow H^{2}(X,T_{X})\to \Ext^2(\Omega_X,\mathcal O_X)$$
induced by the local-to-global spectral sequence for Ext, if $H^0(X,\mathcal{E}xt^1(\Omega_X,\mathcal O_X))=0$ then all deformations are locally trivial, while if $H^1(X,T_X)=0$ then all locally trivial deformations are trivial.

\section{On the rigidity of $\overline{\mathcal{M}}_{g,A[n]}$ and $\overline{M}_{g,A[n]}$}\label{righassett}
Let $\rho:\overline{\mathcal{M}}_{g,A[n]}\rightarrow\overline{\mathcal{M}}_{g,B[n]}$ be a reduction morphism between Hassett moduli stacks. By \cite[Proposition 4.5]{Ha} the morphism $\rho$ contract the boundary divisors $D_{I,J} = \overline{\mathcal{M}}_{0,A_{I}}\times\overline{\mathcal{M}}_{g,A_{J}}$ with $A_{I} = (a_{i_1},...,a_{i_r},1)$, $A_{J} = (a_{j_1},...,a_{j_{n-r}},1)$ and $c = b_{i_1}+...+b_{i_r}\leq 1$ for $2< r\leq n$.\\
By \cite[Remark 4.6]{Ha} the morphism $\rho$ can be factored as a composition of reduction morphisms $\rho = \rho_1\circ...\circ\rho_{k}$ where $\rho_{i}:\overline{\mathcal{M}}_{g,A^{'}[n]}\rightarrow\overline{\mathcal{M}}_{g,B^{'}[n]}$ is the blow-up of $\overline{\mathcal{M}}_{g,B^{'}[n]}$ along the image of a single divisor of type $D_{I,J}$. 
$$
\begin{tikzpicture}[line cap=round,line join=round,>=triangle 45,x=1.0cm,y=0.5263616569431003cm]
\clip(-2.3018779342723006,0.9904801446935761) rectangle (7.20243544600939,6.5);
\draw [shift={(1.52,4.49)}] plot[domain=0.7998899024629484:3.9414825560527413,variable=\t]({1.*1.4638647478507023*cos(\t r)+0.*1.4638647478507023*sin(\t r)},{0.*1.4638647478507023*cos(\t r)+1.*1.4638647478507023*sin(\t r)});
\draw [shift={(-0.41,2.67)}] plot[domain=-2.439335722080786:0.702256931509007,variable=\t]({1.*1.1920570456148483*cos(\t r)+0.*1.1920570456148483*sin(\t r)},{0.*1.1920570456148483*cos(\t r)+1.*1.1920570456148483*sin(\t r)});
\draw (-2.2,5.66)-- (0.66,1.36);
\draw [->] (2.30714,3.58) -- (4.256009389671363,3.581968267236119);
\draw [shift={(6.072629107981221,4.5167109775608445)}] plot[domain=0.7998899024629496:3.9414825560527427,variable=\t]({1.*1.4638647478507019*cos(\t r)+0.*1.4638647478507019*sin(\t r)},{0.*1.4638647478507019*cos(\t r)+1.*1.4638647478507019*sin(\t r)});
\draw [shift={(4.1426291079812225,2.696710977560843)}] plot[domain=-2.439335722080786:0.702256931509007,variable=\t]({1.*1.1920570456148483*cos(\t r)+0.*1.1920570456148483*sin(\t r)},{0.*1.1920570456148483*cos(\t r)+1.*1.1920570456148483*sin(\t r)});
\begin{scriptsize}
\draw [fill=black] (-1.9480760116387195,5.281233164351921) circle (1.0pt);
\draw[color=black] (-1.8528902582159628,5.690816489263434) node {$a_{i_1}$};
\draw [fill=black] (-0.5288427272999972,3.147420883702793) circle (1.0pt);
\draw[color=black] (-0.4392458920187796,3.556755735119648) node {$a_{i_r}$};
\draw [fill=black] (0.7633074270928902,2.880593640760262) circle (1.0pt);
\draw[color=black] (0.8588174882629106,3.2866214624432195) node {$a_{j_1}$};
\draw [fill=black] (2.274104548470395,5.744681764423259) circle (1.0pt);
\draw[color=black] (2.405824530516432,6.150044752813363) node {$a_{j_n-r}$};
\draw[color=black] (3.299354460093897,3.921437003232827) node {$\rho_{i}$};
\draw [fill=black] (4.90150401409298,1.7774129427136711) circle (1.0pt);
\draw[color=black] (4.961942488262911,2.1115373763007548) node {$c$};
\draw [fill=black] (5.250505662646654,3.1367218179075924) circle (1.0pt);
\draw[color=black] (5.30868544600939,3.462208739682898) node {$b_{j_1}$};
\draw [fill=black] (6.801653410541645,5.786129571357346) circle (1.0pt);
\draw[color=black] (6.864583333333334,6.136538039179541) node {$b_{j_{n-r}}$};
\end{scriptsize}
\end{tikzpicture}
$$
We will need the following commutative algebra result.

\begin{Lemma}\label{ca}
Let $R$ be a commutative ring. Given the following commutative diagram of $R$-modules
 \[
  \begin{tikzpicture}[xscale=1.5,yscale=-1.2]
    \node (A0_2) at (2, 0) {$0$};
    \node (A0_3) at (3, 0) {$0$};
    \node (A1_2) at (2, 1) {$C_1$};
    \node (A1_3) at (3, 1) {$C_2$};
        \node (A2_0) at (0, 2) {$0$};
    \node (A2_1) at (1, 2) {$A$};
    \node (A2_2) at (2, 2) {$A_1$};
    \node (A2_3) at (3, 2) {$A_2$};
    \node (A2_4) at (4, 2) {$B$};
    \node (A2_5) at (5, 2) {$0$};
    \node (A3_0) at (0, 3) {$0$};
    \node (A3_1) at (1, 3) {$A$};
    \node (A3_2) at (2, 3) {$B_1$};
    \node (A3_3) at (3, 3) {$B_2$};
    \node (A3_4) at (4, 3) {$B$};
    \node (A3_5) at (5, 3) {$0$};
    \node (A4_2) at (2, 4) {$0$};
    \node (A4_3) at (3, 4) {$0$};
    \path (A3_2) edge [->] node [auto] {$\scriptstyle{}$} (A2_2);
    \path (A2_3) edge [->] node [auto] {$\scriptstyle{}$} (A2_4);
    \path (A3_4) edge [->] node [auto] {$\scriptstyle{}$} (A3_5);
    \path (A2_3) edge [->,swap] node [auto] {$\scriptstyle{\pi_2}$} (A1_3);
    \path (A2_4) edge [->] node [auto] {$\scriptstyle{}$} (A2_5);
    \path (A3_4) edge [-,double distance=1.5pt] node [auto]
{$\scriptstyle{}$} (A2_4);
    \path (A2_1) edge [->] node [auto] {$\scriptstyle{}$} (A2_2);
    \path (A3_0) edge [->] node [auto] {$\scriptstyle{}$} (A3_1);
    \path (A4_3) edge [->] node [auto] {$\scriptstyle{}$} (A3_3);
    \path (A2_2) edge [->] node [auto] {$\scriptstyle{\gamma}$} (A2_3);
    \path (A1_3) edge [->] node [auto] {$\scriptstyle{}$} (A0_3);
    \path (A3_3) edge [->] node [auto] {$\scriptstyle{}$} (A2_3);
    \path (A2_2) edge [->] node [auto] {$\scriptstyle{\pi_1}$} (A1_2);
    \path (A3_1) edge [-,double distance=1.5pt] node [auto]
{$\scriptstyle{}$} (A2_1);
    \path (A4_2) edge [->] node [auto] {$\scriptstyle{}$} (A3_2);
    \path (A3_2) edge [->] node [auto] {$\scriptstyle{}$} (A3_3);
    \path (A1_2) edge [->] node [auto] {$\scriptstyle{}$} (A0_2);
    \path (A2_0) edge [->] node [auto] {$\scriptstyle{}$} (A2_1);
    \path (A3_1) edge [->] node [auto] {$\scriptstyle{}$} (A3_2);
    \path (A3_3) edge [->] node [auto] {$\scriptstyle{}$} (A3_4);
  \end{tikzpicture}
  \]
there exists an isomorphism $\delta:C_{1}\rightarrow C_{2}$.
\end{Lemma}
\begin{proof}
For any $c_1\in C_1$ there exists $a_{1}\in A_{1}$ such that $c_1 = \pi_1(a_1)$. We define $\delta(c_1) = \pi_2(\gamma(a_1))$. It is straightforward to check that $\delta$ is well defined and that it is an isomorphism.
\end{proof}

Now, we are ready to explicit the normal bundle of $\overline{\mathcal{M}}_{g,C[s+1]}\subset\overline{\mathcal{M}}_{g,B[n]}$ in terms of the first psi-class.

\begin{Proposition}\label{normalbun}
Let $\rho:\overline{\mathcal{M}}_{g,A[n]}\rightarrow\overline{\mathcal{M}}_{g,B[n]}$ be a reduction morphism contracting a single boundary divisor $D_{I,J}$ as above, and let $\rho(D_{I,J}) = \overline{\mathcal{M}}_{g,C[s+1]}$ be its image, with $s = n-r$. Then 
$$N_{\overline{\mathcal{M}}_{g,C[s+1]}/\overline{\mathcal{M}}_{g,B[n]}} = (\psi_{1}^{\vee})^{\oplus (r-1)}$$
\end{Proposition}
\begin{proof}
The reduction morphism $\rho:\overline{\mathcal{M}}_{g,A[n]}\rightarrow\overline{\mathcal{M}}_{g,B[n]}$ is the blow-up of $\overline{\mathcal{M}}_{g,B[n]}$ along $\rho(D_{I,J}) = \overline{\mathcal{M}}_{g,C[s+1]}$ with $C = (c,b_{j_1},...,b_{j_{n-r}})$. We identify $\overline{\mathcal{M}}_{g,C[s+1]}$ with the image of the embedding 
$$
\begin{array}{cccc}
 & \overline{\mathcal{M}}_{g,C[s+1]} & \longrightarrow & \overline{\mathcal{M}}_{g,B[n]}\\
 & [C,x_1,...,x_{s+1}] & \longmapsto & [C,x_1,...,x_1,...,x_{s+1}]
\end{array}
$$
Let $[C,x_1,...,x_1,...,x_{s+1}]\in\overline{\mathcal{M}}_{g,C[s+1]}\subset\overline{\mathcal{M}}_{g,B[n]}$ be a point. On the curve $C$ we have the exact sequence
\begin{center}
$0\mapsto\Omega_{C}\rightarrow\Omega_{C}(\sum_{i=1}^{n}x_{i})\rightarrow\bigoplus_{i=1}^{n}\mathcal{O}_{C,x_{i}}\mapsto 0$
\end{center}
Now, since $\Hom(\Omega_{C}(\sum_{i=1}^{n}x_{i}),\mathcal{O}_{C}) = T_{Id}\Aut(C,(x_{1},...,x_n))$ and $(C,(x_{1},...,x_n))$ is stable we have that 
$$\Hom(\Omega_{C}(x_1+...+x_n),\mathcal{O}_{C}) = 0$$ 

Therefore, applying the functor $\mathcal{H}om(-,\mathcal{O}_{C})$ and taking stalks at the point $x_1\in C$ we get the following 
exact sequence
$$0\mapsto \Hom(\Omega_C,\mathcal{O}_{C})\rightarrow\bigoplus_{i=1}^{n}\Ext^{1}(\mathcal{O}_{C,x_{i}},\mathcal{O}_{C})\rightarrow\Ext^{1}(\Omega_{C}(\sum_{i=1}^{n}x_{i}),\mathcal{O}_{C}))\rightarrow\Ext^{1}(\Omega_{C},\mathcal{O}_{C})\mapsto 0$$

On the other hand we have the same exact sequence on $[C,x_1,...,x_1,...,x_{s+1}]$ seen as a point in $\overline{\mathcal{M}}_{g,C[s+1]}$. Therefore we may consider the following diagram
 \[
 \begin{tikzpicture}[xscale=4.4,yscale=-1.2]
  \node (A0_1) at (1, 0) {$(T_{x_1}C)^{\otimes r}/T_{x_{1}}C$};
 \node (A0_2) at (2, 0)
{$N_{\overline{M}_{0,C[s+1]}/\overline{M}_{0,B[n]}|[C,x_{i}]}$};
    \node (A1_0) at (0, 1) {$\Hom(\Omega_C,\mathcal{O}_{C})$};
    \node (A1_1) at (1, 1)
{$\bigoplus_{i=1}^{n}\Ext^{1}(\mathcal{O}_{C,x_{i}},\mathcal{O}_{C})$};
    \node (A1_2) at (2, 1)
{$\Ext^{1}(\Omega_{C}(\sum_{i=1}^{n}x_{i}),\mathcal{O}_{C}))$};
    \node (A1_3) at (3, 1) {$\Ext^{1}(\Omega_{C},\mathcal{O}_{C})$};
    \node (A2_0) at (0, 2) {$\Hom(\Omega_C,\mathcal{O}_{C})$};
    \node (A2_1) at (1, 2)
{$\bigoplus_{i=1}^{s+1}\Ext^{1}(\mathcal{O}_{C,x_{i}},\mathcal{O}_{C})$};
    \node (A2_2) at (2, 2)
{$\Ext^{1}(\Omega_{C}(\sum_{i=1}^{s+1}x_{i}),\mathcal{O}_{C}))$};
    \node (A2_3) at (3, 2) {$\Ext^{1}(\Omega_{C},\mathcal{O}_{C})$};
    \path (A2_1) edge [->] node [auto] {$\scriptstyle{}$} (A2_2);
    \path (A1_0) edge [->] node [auto] {$\scriptstyle{}$} (A1_1);
    \path (A1_3) edge [-,double distance=1.5pt] node [auto]
{$\scriptstyle{}$} (A2_3);
    \path (A0_1) edge [->] node [auto] {$\scriptstyle{}$} node
[rotate=180,sloped] {$\scriptstyle{\widetilde{\ \ \ }}$} (A0_2);
    \path (A1_1) edge [->] node [auto] {$\scriptstyle{}$} (A1_2);
    \path (A2_2) edge [->] node [auto] {$\scriptstyle{}$} (A2_3);
    \path (A1_1) edge [->] node [auto] {$\scriptstyle{}$} (A0_1);
    \path (A2_2) edge [->] node [auto] {$\scriptstyle{\beta}$} (A1_2);
    \path (A2_1) edge [->] node [auto] {$\scriptstyle{\alpha}$} (A1_1);
    \path (A1_0) edge [-,double distance=1.5pt] node [auto]
{$\scriptstyle{}$} (A2_0);
    \path (A1_2) edge [->] node [auto] {$\scriptstyle{}$} (A1_3);
    \path (A1_2) edge [->] node [auto] {$\scriptstyle{}$} (A0_2);
    \path (A2_0) edge [->] node [auto] {$\scriptstyle{}$} (A2_1);
  \end{tikzpicture}
  \]
where the vertical maps are defined as
$$
\begin{array}{cccc}
\alpha: & \bigoplus_{i=1}^{s+1}\Ext^{1}(\mathcal{O}_{C,x_{i}},\mathcal{O}_{C}) & \longrightarrow &\bigoplus_{i=1}^{n}\Ext^{1}(\mathcal{O}_{C,x_{i}},\mathcal{O}_{C})\\
 & (v_1,...,v_{s+1}) & \longmapsto & (v_1,...,v_1,...,v_{s+1})
\end{array}
$$
and 
$$
\begin{array}{cccc}
\beta: & T^{1}Def(C,x_1,...,x_{s+1}) & \longrightarrow & T^{1}Def(C,x_1,...,x_1,...,x_{s+1})\\
 & (\widetilde{C},\widetilde{x}_1,...,\widetilde{x}_{s+1}) & \longmapsto & (\widetilde{C},\widetilde{x}_1,...,\widetilde{x}_1,...,\widetilde{x}_{s+1})
\end{array}
$$
with the identifications $T^{1}Def(C,x_1,...,x_{s+1}) = \Ext^{1}(\Omega_{C}(\sum_{i=1}^{s+1}x_{i}),\mathcal{O}_{C})) = T_{[C,x_i]}\overline{\mathcal{M}}_{g,C[s+1]}$ and $T^{1}Def(C,x_1,...,x_1,...,x_{s+1}) = \Ext^{1}(\Omega_{C}(\sum_{i=1}^{n}x_{i}),\mathcal{O}_{C})) = T_{[C,x_i]}\overline{\mathcal{M}}_{g,B[n]}$.
Furthermore, we have that $\Ext^{1}(\mathcal{O}_{C,x_{i}},\mathcal{O}_{C}) = T_{x_{i}}C$. Hence 
$$\bigoplus_{i=1}^{n}\Ext^{1}(\mathcal{O}_{C,x_{i}},\mathcal{O}_{C})/\bigoplus_{i=1}^{s+1}\Ext^{1}(\mathcal{O}_{C,x_{i}},\mathcal{O}_{C}) = (T_{x_1}C)^{\oplus r}/T_{x_1}C = (T_{x_1}C)^{\oplus (r-1)}$$
By Lemma \ref{ca} we get $N_{\overline{\mathcal{M}}_{g,C[s+1]}/\overline{\mathcal{M}}_{g,B[n]}|[C,x_{i}]} \cong (T_{x_1}C)^{\oplus (r-1)}$, and hence 
$$N_{\overline{\mathcal{M}}_{g,C[s+1]}/\overline{\mathcal{M}}_{g,B[n]}} = (\psi_{1}^{\vee})^{\oplus (r-1)}$$ 
Note that $\codim_{\overline{\mathcal{M}}_{g,B[n]}}(\overline{\mathcal{M}}_{g,C[s+1]}) = n-3-(n-r-2) = r-1$.
\end{proof}
 
Our next aim is to prove a vanishing result for the higher cohomology groups of $\psi$-classes on Hassett space.

\begin{Proposition}\label{psihassett}
Let $\rho:\overline{\mathcal{M}}_{g,A[n]}\rightarrow\overline{\mathcal{M}}_{g,B[n]}$ be a reduction morphism contracting a single boundary divisor $D_{I,J}$ as above, and let $\rho(D_{I,J}) = \overline{\mathcal{M}}_{g,C[s+1]}$ be its image, with $s = n-r$. Assume $H^{j}(\overline{\mathcal{M}}_{g,C[m]},\psi_{i}^{\vee}) = 0$ for any $i=1,...,m$, $j>0$, for any Hassett stack $\overline{\mathcal{M}}_{g,C[m]}$ with $m<n$.
Then 
$$H^{j}(\overline{\mathcal{M}}_{g,A[n]},\psi_{i}^{\vee}) = 0 \Longrightarrow H^{j}(\overline{\mathcal{M}}_{g,B[n]},\psi_{i}^{\vee}) = 0$$
for any $j>0$.
\end{Proposition}
\begin{proof}
Let us write the exceptional divisor as $D_{I,J} = \overline{\mathcal{M}}_{0,C_{1}[r+1]}\times\overline{\mathcal{M}}_{g,C_{2}[s+1]}$. We distinguish two cases: $i > r$ and $i\leq r$.

If $i > r$ then $\psi_{i}^{\vee}\cong \rho^{*}\psi_{i}^{\vee}$. Since $R^{j}\rho_{*}\psi_{i}^{\vee} = 0$ we have that 
$$H^{j}(\overline{\mathcal{M}}_{g,B[n]},\psi_{i}^{\vee}) = H^{j}(\overline{\mathcal{M}}_{g,A[n]},\rho^{*}\psi_{i}^{\vee}) = H^{j}(\overline{\mathcal{M}}_{g,A[n]},\psi_{i}^{\vee}) = 0$$

If $i\leq r$, up to reordering, we have $\rho^{*}\psi_{i}^{\vee} = pr_{2}^{*}\psi_{1}^{\vee}$, where $pr_{2}:D_{I,J}\rightarrow\overline{\mathcal{M}}_{g,C_{2}[s+1]}$ is the second projection. We proceed by induction on the dimension of Hassett stacks. If $n = 3$ then any Hassett space is just a point and $H^{j}(\overline{\mathcal{M}}_{0,3},\psi_{i}^{\vee}) = 0$. Furthermore, we have 
$$H^{j}(D_{I,J},\rho^{*}\psi_{i}^{\vee}) = H^{j}(D_{I,J},pr_{2}^{*}\psi_{1}^{\vee})$$ 
Now, let us consider the exact sequence
$$0\mapsto\psi_{i}^{\vee}\rightarrow\rho^{*}\psi_{i}^{\vee}\rightarrow\rho^{*}\psi_{i|D_{I,J}}^{\vee}\mapsto 0$$
By hypothesis we have $H^{j}(\overline{\mathcal{M}}_{g,A[n]},\psi_{i}^{\vee}) = 0$. Furthermore, $R^{j}pr_{2*}pr_{2}^{*}\psi_{1}^{\vee} = 0$ implies 
$$H^{j}(D_{I,J},\rho^{*}\psi_{i}^{\vee}) = H^{j}(D_{I,J},pr_{2}^{*}\psi_{1}^{\vee}) = H^{j}(\overline{\mathcal{M}}_{g,C_{2}[s+1]},\psi_{1}^{\vee}) = 0$$
by induction hypothesis on $\dim\overline{\mathcal{M}}_{g,C_{2}[s+1]} < \dim\overline{\mathcal{M}}_{g,B[n]}$. Since $R^{j}\rho_{*}\rho^{*}\psi_{i}^{\vee} = 0$ taking the long exact sequence in cohomology we conclude that $H^{j}(\overline{\mathcal{M}}_{g,B[n]},\psi_{i}^{\vee}) = H^{j}(\overline{\mathcal{M}}_{g,B[n]},\rho_{*}\psi_{i}^{\vee}) = 0$.
\end{proof}

We will need the following lemma relating the first order infinitesimal deformations of a stack to the deformations of its blow-up along a smooth substack.

\begin{Lemma}\label{defbu}
Let $\mathcal{X}$ be a smooth stack and $\mathcal{Z}\subseteq\mathcal{X}$ be a smooth substack. Then 
$$T^{1}Def_{(\mathcal{X},\mathcal{Z})} = H^{1}(\mathcal{X},T_{\mathcal{X}}(-\logpol \mathcal{Z})) = H^{1}(Bl_{\mathcal{Z}}\mathcal{X},T_{Bl_{\mathcal{Z}}\mathcal{X}}) = T^{1}Def_{Bl_{\mathcal{Z}}\mathcal{X}}$$
Furthermore, the following diagram
\[
  \begin{tikzpicture}[xscale=2.5,yscale=-1.2]
    \node (A0_0) at (0, 0) {$0$};
    \node (A0_1) at (1, 0) {$T_{\widetilde{\mathcal{X}}}$};
    \node (A0_2) at (2, 0) {$\epsilon^{*}T_{\mathcal{X}}$};
    \node (A0_4) at (4, 0) {$Q$};
    \node (A0_5) at (5, 0) {$0$};
    \node (A1_3) at (3, 1) {$\epsilon^{*}T_{\mathcal{X}|E}$};
    \node (A2_0) at (0, 2) {$0$};
    \node (A2_1) at (1, 2) {$L$};
    \node (A2_2) at (2, 2) {$\epsilon^{*}N_{\mathcal{Z}/\mathcal{X}|E}$};
    \node (A2_4) at (4, 2) {$Q$};
    \node (A2_5) at (5, 2) {$0$};
    \path (A0_4) edge [-,double distance=1.5pt] node [auto] {$\scriptstyle{}$} (A2_4);
    \path (A0_2) edge [->] node [auto] {$\scriptstyle{}$} (A2_2);
    \path (A1_3) edge [->] node [auto] {$\scriptstyle{}$} (A2_2);
    \path (A2_4) edge [->] node [auto] {$\scriptstyle{}$} (A2_5);
    \path (A0_0) edge [->] node [auto] {$\scriptstyle{}$} (A0_1);
    \path (A0_1) edge [->] node [auto] {$\scriptstyle{}$} (A0_2);
    \path (A2_2) edge [->] node [auto] {$\scriptstyle{}$} (A2_4);
    \path (A0_4) edge [->] node [auto] {$\scriptstyle{}$} (A0_5);
    \path (A0_2) edge [->] node [auto] {$\scriptstyle{}$} (A0_4);
    \path (A2_0) edge [->] node [auto] {$\scriptstyle{}$} (A2_1);
    \path (A0_2) edge [->] node [auto] {$\scriptstyle{}$} (A1_3);
    \path (A2_1) edge [->] node [auto] {$\scriptstyle{}$} (A2_2);
 \end{tikzpicture}
  \]
induces an isomorphism $H^{i}(\mathcal{Z},N_{\mathcal{Z}/\mathcal{X}})\cong H^{i}(E,Q)$ for any $i\geq 1$, where $L$ is implicitly defined by requiring the second row to be exact.
\end{Lemma}
\begin{proof}
The argument is the same as for smooth varieties. Let $\widetilde{\mathcal{X}} = Bl_{\mathcal{Z}}\mathcal{X}$ be the blow-up and $\epsilon:\widetilde{\mathcal{X}}\rightarrow \mathcal{X}$ be the blow-up morphism with exceptional divisor $E = \mathbb{P}(N_{\mathcal{Z}/\mathcal{X}})$. 

If $z\in \mathcal{Z}$ then $\epsilon^{-1}(z)\cong\mathbb{P}^{h}$, with $h = \codim(\mathcal{Z})-1$. Furthermore, $L_{|\epsilon^{-1}(z)}\cong\mathcal{O}_{\mathbb{P}^{h}}(-1)$ and $H^{i}(\epsilon^{-1}(z),L_{|\epsilon^{-1}(z)}) = 0$ for any $i$. This implies that $R^{i}\epsilon_{*}L = 0$ for any $i$. Therefore $R^{i}\epsilon_{*}\epsilon^{*}N_{\mathcal{Z}/\mathcal{X}|E}\cong R^{i}\epsilon_{*}Q$ for any $i$.

Now $R^{i}\epsilon_{*}\epsilon^{*}N_{\mathcal{Z}/\mathcal{X}|E}\cong R^{i}\epsilon_{*}\mathcal{O}_{E}\otimes N_{\mathcal{Z}/\mathcal{X}|E}$. Furthermore, $H^{i}(\epsilon^{-1}(z),\mathcal{O}_{\epsilon^{-1}(z)}) = 0$ for any $i > 0$ yields $\epsilon_{*}\mathcal{O}_{E}\otimes N_{\mathcal{Z}/\mathcal{X}|E} = N_{\mathcal{Z}/\mathcal{X}|E}$ and $R^{i}\epsilon_{*}\mathcal{O}_{E}\otimes N_{\mathcal{Z}/\mathcal{X}|E} = 0$ for any $i \geq 1$. Then the Leray spectral sequence for $Q$ degenerates and we get $H^{i}(E,Q) = H^{i}(\mathcal{Z},N_{\mathcal{Z}/\mathcal{X}})$ for any $i\geq 1$. 

Since $R^{i}\epsilon_{*}T_{\widetilde{\mathcal{X}}} = 0$ for $i>0$ we have $\epsilon_{*}T_{\widetilde{\mathcal{X}}} = T_{\mathcal{X}}(-\logpol\mathcal{Z})$. In turns this yields $H^{i}(\widetilde{\mathcal{X}},T_{\widetilde{\mathcal{X}}}) = H^{i}(\mathcal{X},\epsilon_{*}T_{\widetilde{\mathcal{X}}}) = H^{i}(\mathcal{X},T_{\mathcal{X}}(-\logpol\mathcal{Z}))$.
\end{proof}

Now, we are ready to prove the main result of this section.

\begin{Theorem}\label{righas}
If $g = 0$ over any field and if $g\geq 1$ over any field of characteristic zero, we have that $H^{j}(\overline{\mathcal{M}}_{g,A[n]},\psi_{i}^{\vee}) = 0$ for any $i=1,...,n$ and $j>0$. Furthermore, $\overline{\mathcal{M}}_{g,A[n]}$ is rigid.
\end{Theorem}
\begin{proof}
Recall that any Hassett stack $\overline{\mathcal{M}}_{g,A[n]}$ receives a reduction morphism $\rho:\overline{\mathcal{M}}_{g,n}\rightarrow\overline{\mathcal{M}}_{g,A[n]}$. Furthermore, $\rho$ can be factored as a composition of reduction morphisms $\rho = \rho_1\circ...\circ\rho_{k}$ where $\rho_{i}:\overline{\mathcal{M}}_{g,A^{'}[n]}\rightarrow\overline{\mathcal{M}}_{g,B^{'}[n]}$ is the blow-up of $\overline{\mathcal{M}}_{g,B^{'}[n]}$ along the image of a single divisor of type $D_{I,J}$. We deduce the first statement by induction on $k$. At the first step of the induction $k=0$ we have $\overline{\mathcal{M}}_{g,A[n]} = \overline{\mathcal{M}}_{g,n}$.

By Equation $(3.1)$ in the proof of \cite[Theorem 3.1]{FM} we get the statement in the case $g = 0$ over any field. Let us prove the same for $g\geq 1$ over any field of characteristic zero. 

By \cite[Theorem 4]{Kn2} the line bundle $\psi_{n}$ on $\overline{\mathcal{M}}_{g,n}$ is identified with the pull-back of the line bundle $\omega_{\pi}(\Sigma)$ via the isomorphism $c:\overline{\mathcal{M}}_{g,n}\rightarrow\mathcal{U}_{g,n-1}$, where $\pi:\mathcal{U}_{g,n-1}\rightarrow\overline{\mathcal{M}}_{g,n-1}$ is the universal curve over $\overline{\mathcal{M}}_{g,n-1}$. Furthermore, by \cite[Theorem 0.4]{Ke} the $\mathbb{Q}$-line bundle $p_{*}\omega_{\pi}(\Sigma)$ is nef and big, where $p:\mathcal{U}_{g,n-1}\rightarrow U_{g,n-1}$ is the map on the coarse moduli space. 

Since we are over a field of characteristic zero we can apply Kodaira vanishing \cite[Theorem A.1]{Hac} to the line bundle $p_{*}\omega_{\pi}(\Sigma)$. In particular, we get 
$$H^{j}(\overline{\mathcal{M}}_{g,n},\psi_{n}^{\vee}) = H^{j}(\mathcal{U}_{g,n-1},\omega_{\pi}(\Sigma)) = 0$$ 
for $j>0$.

Now, let us consider the second statement. Since by \cite[Theorem 3.1]{FM} we know that $H^{1}(\overline{M}_{0,n},T_{\overline{M}_{0,n}}) = 0$ over any field, and by \cite[Theorem 2.1]{Hac} we have $H^{1}(\overline{\mathcal{M}}_{g,n},T_{\overline{\mathcal{M}}_{g,n}}) = 0$ for $g\geq 1$ over any field of characteristic zero we may proceed by induction on $k$ and prove the second statement for a single morphism $\rho:\overline{\mathcal{M}}_{g,A^{'}[n]}\rightarrow\overline{\mathcal{M}}_{g,B^{'}[n]}$. 

Let $E$ be the exceptional locus of the morphism $\rho$, and let $Z = \rho(E)$. We denote by $T^{1}Def_{(\overline{\mathcal{M}}_{g,B^{'}[n]},Z)}$ the space of first order infinitesimal deformation of the couple $(\overline{\mathcal{M}}_{g,B^{'}[n]},Z)$. 

Then we have the following exact sequence 
$$H^{0}(Z,N_{Z/\overline{\mathcal{M}}_{g,B^{'}[n]}})\rightarrow T^{1}Def_{(\overline{\mathcal{M}}_{g,B^{'}[n]},Z)}\rightarrow T^{1}Def_{\overline{\mathcal{M}}_{g,B^{'}[n]}}\rightarrow H^{1}(Z,N_{Z/\overline{\mathcal{M}}_{g,B^{'}[n]}})$$
By Proposition \ref{normalbun} we have $N_{Z/\overline{\mathcal{M}}_{g,B^{'}[n]}}\cong (\psi_{1}^{\vee})^{\oplus (r-1)}$ and by Proposition \ref{psihassett} we have 
$$H^{1}(\overline{\mathcal{M}}_{g,B^{'}[n]},(\psi_{1}^{\vee})^{\oplus (r-1)}) = 0$$ 
Therefore $H^{1}(Z,N_{Z/\overline{\mathcal{M}}_{g,B^{'}[n]}}) = 0$. By Lemma \ref{defbu} we get
$$T^{1}Def_{(\overline{\mathcal{M}}_{g,B^{'}[n]},Z)} = H^{1}(\overline{\mathcal{M}}_{g,B^{'}[n]},T_{\overline{\mathcal{M}}_{g,B^{'}[n]}}(-\logpol Z)) = H^{1}(\overline{\mathcal{M}}_{g,A^{'}[n]}, T_{\overline{\mathcal{M}}_{g,A^{'}[n]}})$$ 
Furthermore, we have $H^{1}(\overline{\mathcal{M}}_{g,A^{'}[n]}, T_{\overline{\mathcal{M}}_{g,A^{'}[n]}}) = 0$ by induction hypothesis. Therefore $$T^{1}Def_{\overline{\mathcal{M}}_{g,B^{'}[n]}} = H^{1}(\overline{\mathcal{M}}_{g,B^{'}[n]},T_{\overline{\mathcal{M}}_{g,B^{'}[n]}}) = 0$$ 
Finally, by induction we conclude that $H^{1}(\overline{\mathcal{M}}_{g,A[n]},T_{\overline{\mathcal{M}}_{g,A[n]}}) = 0$, that is $\overline{\mathcal{M}}_{g,A[n]}$ is rigid.
\end{proof}

Now, let us consider Hassett moduli spaces $\overline{M}_{g,A[n]}$ such that $g+n\geq 4$. We begin by studying locally trivial deformations.

\begin{Proposition}\label{ltdhas}
If $g+n\geq 4$, over a field of characteristic zero, the coarse moduli space $\overline{M}_{g,A[n]}$ does not have locally trivial first order infinitesimal deformations for any vector of weights $A[n]$.  
\end{Proposition}
\begin{proof}
Without loss of generality we can assume that there exists a reduction morphism $\rho:\overline{M}_{g,n}\rightarrow\overline{M}_{g,A[n]}$ contracting a single boundary divisor $D:=D_{I,J}$. Let $\rho(D) = \overline{\mathcal{M}}_{g,C[s+1]}$ be the image of the exceptional divisor. We have the following diagram
\[
  \begin{tikzpicture}[xscale=3.0,yscale=-1.2]
    \node (A0_0) at (0, 0) {$0$};
    \node (A0_1) at (1, 0) {$T_{\overline{M}_{g,n}}$};
    \node (A0_2) at (2, 0) {$\rho^{*}T_{\overline{M}_{g,A[n]}}$};
    \node (A0_4) at (4, 0) {$Q$};
    \node (A0_5) at (5, 0) {$0$};
    \node (A1_3) at (3, 1) {$\rho^{*}T_{\overline{M}_{g,A[n]}|D}$};
    \node (A2_0) at (0, 2) {$0$};
    \node (A2_1) at (1, 2) {$L$};
    \node (A2_2) at (2, 2) {$\rho^{*}N_{\overline{M}_{g,C[s+1]}/\overline{M}_{g,A[n]}|D}$};
    \node (A2_4) at (4, 2) {$Q$};
    \node (A2_5) at (5, 2) {$0$};
    \path (A0_4) edge [-,double distance=1.5pt] node [auto] {$\scriptstyle{}$} (A2_4);
    \path (A0_2) edge [->] node [auto] {$\scriptstyle{}$} (A2_2);
    \path (A1_3) edge [->] node [auto] {$\scriptstyle{}$} (A2_2);
    \path (A2_4) edge [->] node [auto] {$\scriptstyle{}$} (A2_5);
    \path (A0_0) edge [->] node [auto] {$\scriptstyle{}$} (A0_1);
    \path (A0_1) edge [->] node [auto] {$\scriptstyle{}$} (A0_2);
    \path (A2_2) edge [->] node [auto] {$\scriptstyle{}$} (A2_4);
    \path (A0_4) edge [->] node [auto] {$\scriptstyle{}$} (A0_5);
    \path (A0_2) edge [->] node [auto] {$\scriptstyle{}$} (A0_4);
    \path (A2_0) edge [->] node [auto] {$\scriptstyle{}$} (A2_1);
    \path (A0_2) edge [->] node [auto] {$\scriptstyle{}$} (A1_3);
    \path (A2_1) edge [->] node [auto] {$\scriptstyle{}$} (A2_2);
  \end{tikzpicture}
  \]
where $N_{\overline{M}_{g,C[s+1]}/\overline{M}_{g,A[n]}}=\pi_*N_{\overline{\mathcal{M}}_{g,C[s+1]}/\overline{\mathcal{M}}_{g,A[n]}}$ and $\pi:\overline{\mathcal{M}}_{g,C[s+1]}\rightarrow\overline{M}_{g,C[s+1]}$ is the coarse moduli map. By Lemma \ref{defbu} we have 
$$H^{i}(D,Q) \cong H^{i}(\overline{M}_{g,C[s+1]},N_{\overline{M}_{g,C[s+1]}/\overline{M}_{g,A[n]}})$$
and by Lemma \ref{normalbun} $N_{\overline{\mathcal{M}}_{g,C[s+1]}/\overline{\mathcal{M}}_{g,A[n]}} = (\psi_{1}^{\vee})^{\oplus (r-1)}$. Furthermore by Theorem \ref{righas} we get $H^{1}(\overline{\mathcal{M}}_{g,C[s+1]},\psi_{1}^{\vee}) = 0$ and $H^{1}(\overline{\mathcal{M}}_{g,C[s+1]},N_{\overline{\mathcal{M}}_{g,C[s+1]}/\overline{\mathcal{M}}_{g,A[n]}}) = 0$. Therefore 
$$H^{1}(\overline{M}_{g,C[s+1]},N_{\overline{M}_{g,C[s+1]}/\overline{M}_{g,A[n]}}) = 0$$
as well. Let us consider the exact sequence in cohomology
$$...\rightarrow H^{1}(\overline{M}_{g,n},T_{\overline{M}_{g,n}})\rightarrow H^{1}(\overline{M}_{g,n},\rho^{*}T_{\overline{M}_{g,A[n]}})\rightarrow H^{1}(D,Q)\rightarrow...$$
Since $H^{1}(D,Q) = 0$ we get $H^{1}(\overline{M}_{g,n},\rho^{*}T_{\overline{M}_{g,A[n]}})\cong H^{1}(\overline{M}_{g,n},T_{\overline{M}_{g,n}})$. Furthermore $\rho$ is birational and $\rho_{*}\mathcal{O}_{\overline{M}_{g,n}} = \mathcal{O}_{\overline{M}_{g,A[n]}}$. By the projection formula we have $\rho_{*}\rho^{*}T_{\overline{M}_{g,A[n]}} = T_{\overline{M}_{g,A[n]}}$. Since $R^{i}\rho_{*}\rho^{*}T_{\overline{M}_{g,A[n]}} = 0$ for $i>0$ we conclude
$$H^{1}(\overline{M}_{g,A[n]},T_{\overline{M}_{g,A[n]}}) \cong H^{1}(\overline{M}_{g,A[n]},\rho_{*}\rho^{*},T_{\overline{M}_{g,A[n]}}) \cong H^{1}(\overline{M}_{g,n},\rho^{*}T_{\overline{M}_{g,A[n]}})$$
On the other hand, if $g+n\geq 4$ then 
$$H^{1}(\overline{M}_{g,A[n]},\rho^{*}T_{\overline{M}_{g,A[n]}})\cong H^{1}(\overline{M}_{g,n},T_{\overline{M}_{g,n}})=0$$
by \cite[Theorem 2.3]{Hac}. We conclude that, if $g+n\geq 4$ then $H^{1}(\overline{M}_{g,A[n]},T_{\overline{M}_{g,A[n]}}) = 0$, that is $\overline{M}_{g,A[n]}$ does not have locally trivial first order infinitesimal deformations.
\end{proof}

Finally, we get the following rigidity result for the coarse moduli spaces $\overline{M}_{g,A[n]}$.

\begin{Theorem}\label{rigc}
If $g+n> 4$, over an algebraically closed field of characteristic zero, the coarse moduli space $\overline{M}_{g,A[n]}$ is rigid for any vector of weights $A[n]$.
\end{Theorem}
\begin{proof}
Let $\rho:\overline{M}_{g,n}\rightarrow\overline{M}_{g,A[n]}$ be the reduction morphism. Let $[C,(x_1,...,x_n)]\in \overline{M}_{g,A[n]}$ be a point with $x_{i_1} = ... = x_{i_r}$, and $[\Gamma,(y_1,...,y_n)]\in\rho^{-1}([C,(x_1,...,x_n)])\subset\overline{M}_{g,n}$. Then $\Gamma = \Gamma_1\cup ...\cup \Gamma_k\cup \Gamma^{'}$ where $\Gamma_1,..., \Gamma_k$ are rational components contracted to the point $x_{i_1} = ... = x_{i_r}\in C$, and $\Gamma^{'}$ is isomorphic to $C$. Therefore, we have that $\Aut(C,(x_1,...,x_n))\cong \Aut(\Gamma,(y_1,...,y_n))$.

Now, let us consider the following codimension two, that is of maximal dimension, irreducible components of the singular locus of $\overline{M}_{g,A[n]}$:
\begin{itemize}
\item[-] $Z_i$ for $i = 4,6$, is the codimension two loci parametrizing curves with an elliptic tail having four and six automorphisms respectively;
\item[-] $Y$ is the locus parametrizing reducible curves $E\cup C$ where $E$ is an elliptic curve with a marked point which is fixed by the elliptic involution, and $C$ is a curve of genus $g-1$ with $n-1$ marked points;
\item[-] $W$ is the locus parametrizing reducible curves $C_1\cup C_2$ where $C_1$ and $C_2$ are of genus two and $g-2$ respectively, the marked points are on $C_2$, and $C_1\cap C_2$ is a fixed point of the hyperelliptic involution on $C_1$.
\end{itemize} 

By the observation on the automorphism groups of the curves in the first part of the proof and \cite[Proposition 5.7]{FM}, we have that when $g+n > 4$ the only codimension two irreducible components of $\Sing(\overline{M}_{g,A[n]
})$ are $Z_4, Z_6, Y$ and $W$. Furthermore, each component contains dense open subsets, denoted by a superscript zero, with complement of codimension at least two such that $\overline{M}_{g,A[n]
}$ has transversal $A_1$ singularities along $Z^0_4, Y^0$ and $W^0$, and transversal $\frac{1}{3}(1,1)$ singularities along $Z^0_6$. 

By Proposition \ref{ltdhas} we know that $\overline{M}_{g,A[n]}$ does not have locally trivial deformations. Therefore, it is enough to prove that $H^{0}(\overline{M}_{g,A[n]},\mathcal{E}xt^{1}(\Omega_{\overline{M}_{g,A[n]}},\mathcal{O}_{\overline{M}_{g,A[n]}}))=0$.\\
Note that $\mathcal{E}xt^{1}(\Omega_{\overline{M}_{g,A[n]}},\mathcal{O}_{\overline{M}_{g,A[n]}})$ is a coherent sheaf supported on $\Sing(\overline{M}_{g,A[n]})$. By \cite[Lemmas 2.4, 2.5]{Fan} there are no sections of $\mathcal{E}xt^{1}(\Omega_{\overline{M}_{g,A[n]}},\mathcal{O}_{\overline{M}_{g,A[n]}})$ supported on the components of $\Sing(\overline M_{g,A[n]})$ of codimension greater than two.

Now, to conclude it is enough to argue as in \cite[Theorem 5.13]{FM} by using the deformation theory of varieties with transversal $A_1$ and $\frac{1}{3}(1,1)$ singularities developed in \cite[Sections 5.4, 5.5]{FM}.
\end{proof}

\begin{Remark}
Note that by Remark \ref{n<3} we have that $\overline{M}_{1,A[2]}\cong\overline{M}_{1,2}$ for any weight data. Therefore, by \cite[Theorem 4.8]{FM} $\overline{M}_{1,A[2]}$ does not have locally trivial deformations, while its family of first order infinitesimal deformations is non-singular of dimension six and the general deformation is smooth.
\end{Remark}

\subsection{Automorphisms of Hassett spaces in arbitrary characteristic}\label{Saut}
In this section we apply the rigidity results in Section \ref{righassett} to extend the main results on the automorphism groups of Hassett spaces in \cite{MM} over an arbitrary field. In order to lift automorphisms from zero to positive characteristic we will use the techniques developed in \cite[Section 1]{FM} considering the ring $W(K)$ of Witt vectors over $K$, see \cite{Wi} for details.

For our purposes it is enough to keep in mind that $W(K)$ is a discrete valuation ring with a closed point $x\in\Spec(W(K))$ with residue field $K$, and a generic point $\xi\in\Spec(W(K))$ with residue field of characteristic zero.

Note that not all permutations of the markings define an automorphism of the space $\overline{M}_{g,A[n]}$. Indeed in order to define an automorphism, permutations have to preserve the weight data in a suitable sense.

For instance, consider Hassett space $\overline{M}_{1,A[4]}$ with weights $(1,1/3,1/3,1/3)$ and the divisor parametrizing reducible curves $C_{1}\cup C_{2}$, where $C_{1}$ has genus zero and markings $(1,1/3,1/3)$, and $C_{2}$ has genus one and marking $1/3$. After the transposition $1\leftrightarrow 4$ the genus zero component has markings $(1/3,1/3,1/3)$, so it is contracted. This means that the transposition $1\leftrightarrow 4$ induces just a birational automorphism of $\overline{M}_{1,A[4]}$ contracting a divisor on a codimension two subscheme. This example leads us to the following definition.

\begin{Definition}\label{atrans}
A transposition $i\leftrightarrow j$ of two marked points is \textit{admissible} if and only if for any $h_{1},...,h_{r}\in\{1,...,n\}\setminus\{i,j\}$, with $r\geq 2$,
$$a_{i}+\sum_{k=1}^r a_{h_{k}}\leq 1 \iff
a_{j}+\sum_{k=1}^r a_{h_{k}}\leq 1.$$
We denote by $\mathcal{S}_{A[n]}\subseteq S_n$ the subgroup of permutations generated by admissible transpositions.
\end{Definition}

We begin by taking into account Hassett spaces appearing in Construction \ref{kblu}. 

\begin{Theorem}\label{aut0}
Let $K$ be any field. For Hassett spaces appearing in Construction \ref{kblu} we have that if $2\leq r\leq n-4$ then:
\begin{itemize}
\item[-] $\Aut(\overline{M}_{0,A_{r,1}[n]}^K)\cong S_{n-r}\times S_r$, 
\item[-] $\Aut(\overline{M}_{0,A_{r,s}[n]}^K)\cong S_{n-r-1}\times S_r$, if $1 < s < n-r-2$,
\item[-] $\Aut(\overline{M}_{0,A_{r,n-r-2}[n]}^K)\cong S_{n-r-1}\times S_{r+1}$, 
\end{itemize}
and if $r = n-3$ then $s =1$, $\overline{M}_{0,A_{n-3,1}[n]}^K\cong\overline{M}_{0,n}^K$, and $\Aut(\overline{M}_{0,A_{n-3,1}[n]}^K)\cong S_n$ for any $n\geq 5$.\\
Finally, if $\ch(K)=0$ for the Losev-Manin moduli space we have:
$$\Aut(\overline{M}_{0,A_{1,n-3}[n]}^K)\cong (K^{*})^{n-3}\rtimes (S_2\times S_{n-2})$$
\end{Theorem}
\begin{proof}
Let $K$ be a field of characteristic zero, and let $\overline{K}$ be its algebraic closure. By \cite[Proposition 1]{FM} there exists an injective morphism of groups
$$\chi:\Aut(\overline{M}_{0,A[n]}^K)\rightarrow\Aut(\overline{M}_{0,A[n]}^{\overline{K}})$$
for any weight data $A[n]$. To conclude that, for Hassett spaces appearing in the statement, $\chi$ is surjective it is enough to apply \cite[Theorem 1]{MM}.

Now, let $K_p$ be a field of characteristic $p>0$, and $W(K_p)$ the ring of Witt vectors of $K_p$ with residue field $K$ of characteristic zero. By Theorem \ref{righas} we have that $H^1(\overline{M}_{0,A[n]}^{K_p}, T_{\overline{M}_{0,A[n]}^{K_p}})=0$. Furthermore, if $r\neq 1$ and $s\neq n-3$ by the first part of the proof we know that $\Aut(\overline{M}_{0,A_{r,s}[n]}^K)$ is finite. Since $\ch(K)=0$ this yields $H^0(\overline{M}_{0,A[n]}^{K}, T_{\overline{M}_{0,A[n]}^{K}})=0$. Now, by \cite[Theorem 1.6]{FM} we get an injective morphism of groups
$$\chi_p:\Aut(\overline{M}_{0,A_{r,s}[n]}^{K_p})\rightarrow\Aut(\overline{M}_{0,A_{r,s}[n]}^K)$$
which by the first part of proof is surjective as well.
\end{proof}

Now, let us move to the case $g\geq 1$.

%\begin{Remark}\label{not2}
%The factorization results in \cite[Proposition 2.5]{MM} are essentially inherited from \cite[Theorem 0.9]{GKM} which works over any algebraically closed field of characteristic different from two. Therefore, the results an the automorphisms of $\overline{M}_{g,A[n]}$ for $g\geq 1$ in \cite[Theorem 2]{MM} hold over any algebraically closed field of characteristic different from two as well.
%\end{Remark}

\begin{Theorem}\label{autg}
Let $K$ be a field with $\ch(K)\neq 2$. If $g\geq 1$ and $2g-2+n\geq 3$ then
$$\Aut(\overline{\mathcal{M}}_{g,A[n]}^K)\cong\Aut(\overline{M}_{g,A[n]}^K)\cong \mathcal{S}_{A[n]}.$$
Furthermore, if $K$ is algebraically closed with $\ch(K)\neq 2, 3$ then we have
\begin{itemize}
\item[-] $\Aut(\overline{M}_{1,A[2]}^K)\cong (K^{*})^{2}$ while $\Aut(\overline{\mathcal{M}}_{1,A[2]}^K)$ is trivial,
\item[-] $\Aut(\overline{M}_{1,A[1]}^K)\cong PGL(2,K)$ while $\Aut(\overline{\mathcal{M}}_{1,A[1]}^K)\cong K^{*}$.
\end{itemize}
\end{Theorem}
\begin{proof}
First of all, note that by Remark \ref{n<3} we have $\overline{M}_{1,A[2]}^K\cong\overline{M}_{1,2}^K$, and $\overline{M}_{1,A[1]}^K\cong\overline{M}_{1,1}^K$ for any weight data. Therefore, if $g = 1$ and $n=1, 2$ the statement follows from \cite[Proposition 4.4]{FM} and \cite[Remark A.4]{FM}.

Now, let $(g,n)\notin\{(1,1),(1,2)\}$, and let $K$ be an algebraically closed field with $\ch(K)\neq 2$. In this case the statement follows form \cite[Theorem 2]{MM}.

If $K$ is a field with $\ch(K)\neq 2$, and $\overline{K}$ is its algebraic closure then by \cite[Proposition 1]{FM} we have an injective morphism of groups
$$\chi:\Aut(\overline{M}_{g,A[n]}^K)\rightarrow\Aut(\overline{M}_{g,A[n]}^{\overline{K}})$$
for any weight data $A[n]$. To conclude it is enough to observe that by the first part of the proof $\Aut(\overline{M}_{g,A[n]}^{\overline{K}})\cong \mathcal{S}_{A[n]}$, and any permutation in $\mathcal{S}_{A[n]}$ induces an automorphism of $\Aut(\overline{M}_{g,A[n]}^{\overline{K}})$ as well.

Finally, by \cite[Proposition 1.7]{FM} we have an injective morphism of groups
$$\widetilde{\chi}:\Aut(\overline{\mathcal{M}}_{g,A[n]}^K)\rightarrow\Aut(\overline{M}_{g,A[n]}^K)\cong \mathcal{S}_{A[n]}$$

An admissible transposition defines an automorphism of $\overline{\mathcal{M}}_{g,A[n]}$. Indeed, the contraction of a rational tail with three special points, that is a nodally attached rational tail with two marked points, does not affect either the coarse moduli space or the stack because it is a bijection on points and preserves the automorphism groups of the objects. Therefore, the morphism $\widetilde{\chi}$ is surjective as well.
\end{proof}

\section{Hassett spaces with weights summing to two and the Segre cubic}\label{Ssegre}
In \cite[Section 2.1.2]{Ha} B. Hassett considers a natural variation on the moduli problem of weighted pointed rational stable curves by considering weights of the type $\widetilde{A}[n]= (a_1,...,a_n)$ such that $a_{i}\in\mathbb{Q}$, $0< a_i\leq 1$ for any $i = 1,...,n$, and 
$$\sum_{i=1}^{n}a_{i} =2$$
By \cite[Section 2.1.2]{Ha} we may construct an explicit family of such weighted curves $\mathcal{C}(\widetilde{A})\rightarrow\overline{M}_{0,n}$ over $\overline{M}_{0,n}$ as an explicit blow-down of the universal curve over $\overline{M}_{0,n}$.\\
Furthermore, if $a_i<1$ for any $i=1,...,n$ we may interpret the geometric invariant theory quotient $(\mathbb{P}^1)^n\quot SL_2$ with respect to the linearization $\mathcal{O}(a_1,...,a_n)$ as the moduli space $\overline{M}_{0,\widetilde{A}[n]}$ associated to the family $\mathcal{C}(\widetilde{A})$.

In this section we will show that Hassett spaces with weights summing to two can have non-trivial first order infinitesimal deformations by considering a specific example. 

Let us consider the weight data 
\begin{equation}\label{weights}
A[6] = (1,1/3,1/3,1/3,1/3,1/3), \: \widetilde{A}[6] = (1/3,1/3,1/3,1/3,1/3,1/3)
\end{equation}
and the reduction morphism
\begin{equation}\label{smallres}
\rho:\overline{M}_{0,A[6]}\rightarrow\overline{M}_{0,\widetilde{A}[6]}
\end{equation}
By \cite[Section 6.2]{Ha} the moduli space $\overline{M}_{0,A[6]}$ is the blow-up of $\mathbb{P}^3$ at five points $p_i\in\mathbb{P}^3$ in linear general position. 

Let $|\mathcal{I}_{p_1,...,p_5}(2)|\subset |\mathcal{O}_{\mathbb{P}^3}(2)|$ be the linear system of quadrics in $\mathbb{P}^3$ through the $p_i$'s. Note that $|\mathcal{I}_{p_1,...,p_5}(2)|$ induces a rational map $\phi:\mathbb{P}^3\dasharrow\mathbb{P}^4$ whose image is a hypersurface $\mathcal{S}\subset\mathbb{P}^4$ of degree $\deg(\mathcal{S}) = 2^3-5 =3$. 

Since the base locus of $\phi$ consists exactly of the $p_i$'s we get a morphism $\widetilde{\phi}:\overline{M}_{0,A[6]}\rightarrow\mathbb{P}^4$ fitting in the following commutative diagram
\[
  \begin{tikzpicture}[xscale=3.7,yscale=-1.6]
    \node (A0_0) at (0, 0) {$\overline{M}_{0,A[6]}$};
    \node (A1_0) at (0, 1) {$\mathbb{P}^3$};
    \node (A1_1) at (1, 1) {$\mathcal{S}\subset\mathbb{P}^4$};
    \path (A0_0) edge [->,swap]node [auto] {$\scriptstyle{\pi_{p_i}}$} (A1_0);
    \path (A1_0) edge [->, dashed]node [auto] {$\scriptstyle{\phi}$} (A1_1);
    \path (A0_0) edge [->]node [auto] {$\scriptstyle{\widetilde{\phi}}$} (A1_1);
  \end{tikzpicture}
  \]
where $\pi_{p_i}:\overline{M}_{0,A[6]}\rightarrow\mathbb{P}^3$ is the blow-up of the $p_i$'s. Note that the only curves contracted by $\widetilde{\phi}$ are the strict transforms $\widetilde{L}_{i,j}$ of the then lines $L_{i,j}=\left\langle p_i,p_j\right\rangle$. Therefore $\widetilde{\phi}$ is a small contraction, and since 
$$N_{\widetilde{L}_{i,j}/\overline{M}_{0,A[6]}}\cong \mathcal{O}_{\widetilde{L}_{i,j}}(-1)\oplus \mathcal{O}_{\widetilde{L}_{i,j}}(-1)$$ 
we conclude that $\mathcal{S} = \widetilde{\phi}(\overline{M}_{0,A[6]})\subset\mathbb{P}^4$ is a cubic surface singular at the ten points $q_{i,j}=\widetilde{\phi}(\widetilde{L}_{i,j})$, and these ten points are nodes of $\mathcal{S}$, that is ordinary singularities with $\mult_{q_{i,j}}(\mathcal{S}) = 2$. Note that in dimension greater than two nodes are not finite quotient singularities, therefore they may contribute to the infinitesimal deformations of $\mathcal{S}$.

A cubic hypersurface in $\mathbb{P}^4$ whose singular locus consists of ten nodes is, up to a change of coordinates, a Segre cubic \cite[Proposition 2.1]{Do15}, that is the hypersurface defined by the equations
$$
\left\lbrace\begin{array}{l}
x_0^3+x_1^3+x_2^3+x_3^3+x_4^3+x_5^3=0,\\ 
x_0+x_1+x_2+x_3+x_4+x_5=0.
\end{array}\right. 
$$
where $[x_0:...:x_5]$ are homogeneous coordinates on $\mathbb{P}^5$. The ten nodes are located at the points conjugate to $[1:1:1:-1:-1:-1]$ under the action of $S_6$ permuting the coordinates.

The Segre cubic is a very interesting and peculiar variety in classical algebraic geometry, indeed its Hessian is the Barth-Nieto quintic, its intersection with a hyperplane of the form $x_i = 0$ is the Clebsch cubic surface, while its intersection with a hyperplane of the type $x_i = x_j$ is the Cayley's nodal cubic surface, and finally it is dual to the Igusa quartic $3$-fold in $\mathbb{P}^4$ \cite{Do15}.

The discussion above shows that we may interpret the morphism $\widetilde{\phi}$ as the reduction morphism $\rho:\overline{M}_{0,A[6]}\rightarrow\overline{M}_{0,\widetilde{A}[6]}$, and Hassett space $\overline{M}_{0,\widetilde{A}[6]}$ as the Segre cubic $\mathcal{S}\subset\mathbb{P}^4$. 

\begin{Proposition}\label{AutSeg}
Let $\mathcal{S}\subset\mathbb{P}^4$ be the Segre cubic. Then $\Aut(\mathcal{S})\cong S_6$ is the permutation group on six elements.
\end{Proposition}
\begin{proof}
Let us consider the weights $A[6]$ and $\widetilde{A}[6]$ in (\ref{weights}), and the small resolution $\rho:\overline{M}_{0,A[6]}\rightarrow\overline{M}_{0,\widetilde{A}[6]}$ in (\ref{smallres}). Since $\Sing(\mathcal{S})$ cosists of ten points which are ordinary double points we may resolve the singularities of $\mathcal{S}$ just by blowing-up these ten points. Let $f:X\rightarrow\mathcal{S}$ be the blow-up. 

Now, by Construction \ref{kblu} the blow-up of $\overline{M}_{0,A[6]}$ along the strict transforms of the ten lines through two of the five points is isomorphic to $\overline{M}_{0,6}$, and we have a reduction morphism $\overline{\rho}:\overline{M}_{0,6}\rightarrow\overline{M}_{0,A[6]}$.

Note that the morphism $\rho\circ\overline{\rho}:\overline{M}_{0,6}\rightarrow\mathcal{S}$ maps the exceptional divisor $E_{i,j}$ over the strict transform $\widetilde{L}_{i,j}$ to the singular point $q_{i,j}\in\mathcal{S}$. Therefore, by the universal property of the blow-up \cite[Proposition 7.4]{Har} there exists a unique morphism $\xi:\overline{M}_{0,6}\rightarrow X$ such that the following diagram
\[
\begin{tikzpicture}[xscale=2.5,yscale=-1.3]
    \node (A0_0) at (0, 0) {$\overline{M}_{0,6}$};
    \node (A0_1) at (1, 0) {$X$};
    \node (A1_0) at (0, 1) {$\overline{M}_{0,A[6]}$};
    \node (A1_1) at (1, 1) {$\mathcal{S}$};
    \path (A0_0) edge [->]node [auto] {$\scriptstyle{\xi}$} (A0_1);
    \path (A1_0) edge [->]node [auto] {$\scriptstyle{\rho}$} (A1_1);
    \path (A0_1) edge [->]node [auto] {$\scriptstyle{f}$} (A1_1);
    \path (A0_0) edge [->]node [auto,swap] {$\scriptstyle{\overline{\rho}}$} (A1_0);
  \end{tikzpicture} 
\]
commutes. Now, note that since $X$ is smooth $\xi$ can not be a small contraction. On the other hand, since $\rho$ is small and $\xi$ must map the exceptional divisor $E_{i,j}$ onto the exceptional divisor $F_{i,j}\subset X$ over $q_{i,j}\in \mathcal{S}$ the morphism $\xi$ can not be a divisorial contraction either. Therefore, $\xi$ is an isomorphism and $X\cong \overline{M}_{0,6}$. 

Now, let $\phi\in\Aut(\mathcal{S})$ be an automorphism. Then $\phi$ must preserve the set $\Sing(\mathcal{S})$ of the ten singular points. Therefore, by \cite[Corollary 7.15]{Har} $\phi$ lifts to an automorphism $\widetilde{\phi}$ of $X\cong\overline{M}_{0,6}$, and we get an injective morphism of groups
$$\chi:\Aut(\mathcal{S})\rightarrow\Aut(\overline{M}_{0,6})$$
Now, to conclude it is enough to recall that $\mathcal{S}\cong\overline{M}_{0,\widetilde{A}[6]}$ is the geometric invariant theory quotient $(\mathbb{P}^1)^6\quot SL_2$ with respect to the symmetric linearization $\mathcal{O}(1/3,...,1/3)$, hence $S_6$ acts on $\mathcal{S}\cong\overline{M}_{0,\widetilde{A}[6]}$ by permuting the marked points, and that $\Aut(\overline{M}_{0,6})\cong S_6$ \cite[Theorem 3]{BM}, \cite[Theorem 3.10]{Ma}, \cite[Theorem 1.1]{FM}. Hence $\chi$ is surjective as well.
\end{proof}

Now, we are ready to study the infinitesimal deformations of the Segre cubic.

\begin{Theorem}\label{segre}
Hassett moduli space $\overline{M}_{0,\widetilde{A}[6]}$ with weights $a_1 = ... = a_6 = \frac{1}{3}$ does not have locally trivial deformations, while its family of first order infinitesimal deformations is non-singular of dimension ten and the general deformation is smooth.
\end{Theorem}
\begin{proof}
All along the proof we will identify $\overline{M}_{0,\widetilde{A}[6]}$ with the Segre cubic $\mathcal{S}\subset\mathbb{P}^4$. The first order infinitesimal deformations of $\mathcal{S}$ are parametrized by the group $\Ext^{1}(\Omega_{\mathcal{S}},\mathcal{O}_{\mathcal{S}})$. The sheaf $\mathcal{E}xt^{1}(\Omega_{\mathcal{S}},\mathcal{O}_{\mathcal{S}})$ is supported on the singularities of $\mathcal{S}$, and since $\mathcal{S}$ has isolated singularities $\mathcal{E}xt^{1}(\Omega_{\mathcal{S}},\mathcal{O}_{\mathcal{S}})$ can be computed separately for each singular point.

Recall that $\mathcal{S}$ is singular at ten nodes, let $p\in \mathcal{S}$ be one of these nodes. Then, \'etale locally, in a neighborhood of $p$ the Segre cubic $\mathcal{S}$ is isomorphic to an \'etale neighborhood of the singularity
$$S = \{f(x,y,z,w) = x^2w+xy-zw = 0\}\subset\mathbb{A}^{4}$$

Indeed, note that the partial derivatives of $f$ vanish simultaneously just at $(0,0,0,0)\in\mathbb{A}^4$, and $\frac{\partial f}{\partial x\partial y}=1$. Furthermore, the projective tangent cone of $S$ at $(0,0,0,0)$ is a smooth quadric surface in $\mathbb{P}^3$. This means that $S$ has an ordinary singularity of multiplicity two at the origin.

Let $R = K[x,y,z,w]/(x^2w+xy-zw)$, and let us consider the free resolution
    \[
  \begin{tikzpicture}[xscale=1.5,yscale=-1.2]
    \node (A0_0) at (0, 0) {$0\mapsto R$};
    \node (A0_1) at (1, 0) {$R^{\oplus 4}$};
    \node (A0_2) at (2, 0) {$\Omega_R\mapsto 0$};
    \path (A0_0) edge [->]node [auto] {$\scriptstyle{\psi_J}$} (A0_1);
    \path (A0_1) edge [->]node [auto] {$\scriptstyle{}$} (A0_2);
  \end{tikzpicture}
  \]
of $\Omega_R$, where $\psi_J$ is the matrix of the partial derivatives of $f = x^2w+xy-zw$. Therefore, we get 
$$
\begin{array}{l}
\Ext^{1}(\Omega_{R},R)\cong R/\im(\psi_J^{t})\cong K[x,y,z,w]/(x^2w+xy-zw, 2xw+y, x, -w, x^2-z)\cong K,\\ 
\Ext^{2}(\Omega_{R},R)=0.
\end{array} 
$$

Now, let us consider the exact sequence
$$0\mapsto \mathcal{O}_\mathcal{S}(-3)\rightarrow \Omega^{1}_{\mathbb{P}^4|\mathcal{S}}\rightarrow\Omega^{1}_{\mathcal{S}}\mapsto 0$$
by applying $\mathcal{H}om_{\mathcal{O}_\mathcal{S}}(-,\mathcal{O}_\mathcal{S})$ we get
\begin{equation}\label{ex1}
0\mapsto T_\mathcal{S}\rightarrow T_{\mathbb{P}^4|\mathcal{S}}\rightarrow \mathcal{O}_\mathcal{S}(3)\rightarrow \mathcal{E}xt^{1}(\Omega_{\mathcal{S}},\mathcal{O}_{\mathcal{S}})\mapsto 0
\end{equation}

Therefore, $H^i(\mathcal{S},T_\mathcal{S}) = 0$ for $i\geq 2$, and by taking Euler-Poincar\'e characteristics we have
$$\chi(T_\mathcal{S})+\chi(\mathcal{O}_\mathcal{S}(3)) = \chi(T_{\mathbb{P}^4|\mathcal{S}})+\chi(\mathcal{E}xt^{1}(\Omega_{\mathcal{S}},\mathcal{O}_{\mathcal{S}}))$$
and since $\chi(T_{\mathbb{P}^4|\mathcal{S}}) = 24$, and $\chi(\mathcal{O}_\mathcal{S}(3)) = \binom{4+3}{3}-1 = 34$ we get
\begin{equation}\label{eqchi}
\chi(T_\mathcal{S}) = \chi(\mathcal{E}xt^{1}(\Omega_{\mathcal{S}},\mathcal{O}_{\mathcal{S}}))-10
\end{equation}

Note that we may interpret $\chi(\mathcal{E}xt^{1}(\Omega_{\mathcal{S}},\mathcal{O}_{\mathcal{S}})) = h^{0}(\mathcal{S},\mathcal{E}xt^{1}(\Omega_{\mathcal{S}},\mathcal{O}_{\mathcal{S}})) = \sum_{i=1}^{10}\tau_{p_i}$, where $\tau_{p_i}$ is the Tyurina number of the node $p_i\in \mathcal{S}$, that is the rank at $p_i$ of the skyscraper sheaf $\mathcal{E}xt^{1}(\Omega_{\mathcal{S}},\mathcal{O}_{\mathcal{S}})$. Note that since $\Ext^{1}(\Omega_{R},R)\cong K$ we have $\tau_{p_i} = 1$ for any $i = 1,...,10$, and hence Equation (\ref{eqchi}) yields
\begin{equation}\label{eqfin}
h^0(\mathcal{S},T_\mathcal{S})-h^1(\mathcal{S},T_\mathcal{S}) = 0
\end{equation}

Now, by Proposition \ref{AutSeg} we have that $\Aut(\mathcal{S})\cong S_6$. Therefore, $\mathcal{S}$ does not have infinitesimal automorphisms and $H^0(\mathcal{S},T_\mathcal{S})=0$. This last fact together with Equation (\ref{eqfin}) forces $H^1(\mathcal{S},T_\mathcal{S}) = 0$ as well.

So the sequence
$$H^{1}(\mathcal{S},T_{\mathcal{S}})\rightarrow\Ext^{1}(\Omega_{\mathcal{S}},\mathcal{O}_{\mathcal{S}})\rightarrow H^{0}(\mathcal{S},\mathcal{E}xt^{1}(\Omega_{\mathcal{S}},\mathcal{O}_{\mathcal{S}}))\rightarrow H^{2}(\mathcal{S},T_{\mathcal{S}})$$
yields 
$$\dim_{K}(\Ext^{1}(\Omega_{\mathcal{S}},\mathcal{O}_{\mathcal{S}}))= h^{0}(\mathcal{S},\mathcal{E}xt^{1}(\Omega_{\mathcal{S}},\mathcal{O}_{\mathcal{S}}))=10$$

Finally, to compute the dimension of the obstruction space $\Ext^{2}(\Omega_{\mathcal{S}},\mathcal{O}_{\mathcal{S}})$ we use the local-to-global Ext spectral sequence 
$$
H^i(\mathcal{S},\mathcal{E}xt^j(\Omega_{\mathcal{S}},\mathcal{O}_{\mathcal{S}}))\Rightarrow \Ext^{i+j}(\Omega_{\mathcal{S}},\mathcal{O}_{\mathcal{S}})
$$

Note that $H^1(\mathcal{S},\mathcal{E}xt^1(\Omega_{\mathcal{S}},\mathcal{O}_{\mathcal{S}}))=0$ because $\mathcal{E}xt^1(\Omega_{\mathcal{S}},\mathcal{O}_{\mathcal{S}})$ is supported on a zero dimensional scheme. Moreover, by (\ref{ex1}) we have $H^2(\mathcal{S},\mathcal{E}xt^0(\Omega_{\mathcal{S}},\mathcal{O}_{\mathcal{S}}))= H^{2}(\mathcal{S},T_{\mathcal{S}}) =0$. Finally, $\Ext^{2}(\Omega_{R},R)=0$ yields $H^0(\mathcal{S},\mathcal{E}xt^2(\Omega_{\mathcal{S}},\mathcal{O}_{\mathcal{S}}))= 0$ as well.
\end{proof}


\begin{thebibliography}{9999999}
\bibitem[AM16]{AM} \bibaut{C. Araujo, A. Massarenti}, \textit{Explicit log Fano structures on blow-ups of projective spaces}, Proceedings London Mathematical Society, 113, 4, 445-473, 2016.
\bibitem[BM13]{BM} \bibaut{A. Bruno, M. Mella}, \textit{The automorphism group of $\overline{M}_{0,n}$}, J. Eur. Math. Soc. Volume 15, Issue 3, 2013, 949-968. 
\bibitem[Do15]{Do15} \bibaut{I. Dolgachev}, \textit{Corrado Segre and nodal cubic threefolds}, 2015, \arXiv{1501.06432v1}.
\bibitem[Fa95]{Fan} \bibaut{B. Fantechi}, \textit{Deformation of Hilbert Schemes of Points on a Surface}, Compositio Mathematica 98, 1995, 205-217.
\bibitem[FM16]{FM} \bibaut{B. Fantechi, A. Massarenti}, \textit{On the rigidity of moduli of curves in arbitrary characteristic}, International Mathematics Research Notices, DOI 10.1093/imrn/rnw105, 2016.
%\bibitem[Fi87]{Fi87} \bibaut{H. Finkelnberg}, \textit{Small resolutions of the Segre cubic}, Indagationes Mathematicae, Vol. 90, Issue 3, 1987, 261-277.
%\bibitem[GKM02]{GKM} \bibaut{A. Gibney, S. Keel, I. Morrison}, \textit{Towards the ample cone of $\overline{M}_{g,n}$}, J. Amer. Math. Soc. 15, 2002, no. 2, 273-294. 
\bibitem[Hac08]{Hac} \bibaut{P. Hacking}, \textit{The moduli space of curves is rigid}, Algebra Number Theory 2, 2008, no. 7, 809-818.
\bibitem[Har77]{Har} \bibaut{R. Hartshorne}, \textit{Algebraic geometry}, Graduate Texts in Mathematics, no. 52, Springer-Verlag, New York-Heidelberg, 1977.
\bibitem[Has03]{Ha} \bibaut{B. Hassett}, \textit{Moduli spaces of weighted pointed stable curves}, Advances in Mathematics, 173, 2003, Issue 2, 316-352.
\bibitem[Ill71]{Ill} \bibaut{L. Illusie}, \textit{Complexe Cotangent et D\'eformations I}, Springer Lecture Notes 239, Springer, 1971.
\bibitem[Ka93]{Ka} \bibaut{M. Kapranov}, \textit{Veronese curves and Grothendieck-Knudsen moduli space $\overline{M}_{0,n}$}, J. Algebraic Geom, 2, 1993, no. 2, 239-262.
\bibitem[Ke99]{Ke} \bibaut{S. Keel}, \textit{Basepoint freeness for nef and big line bundles in positive characteristic}, Ann. of Math, 2, 149, 1999, no. 1, 253-286.
\bibitem[KM97]{KM} \bibaut{S. Keel, S. Mori}, \textit{Quotients by groupoids}, Ann. of Math, 2, 145, 1997, 193-213.
\bibitem[Kn83]{Kn2} \bibaut{F. F. Knudsen}, \textit{The projectivity of the moduli space of stable curves II. The stacks $M_{g,n}$}, Math. Scand. 52, 1983, no. 2, 161-199.
\bibitem[LM00]{LM} \bibaut{A. Losev, Y. Manin}, \textit{New moduli spaces of pointed curves and pencils of flat connections}, Michigan Math. J, vol. 48, Issue 1, 2000, 443-472.
\bibitem[Ma16]{Ma16} \bibaut{A. Massarenti}, \textit{On the biregular geometry of Fulton-MacPherson configuration spaces}, 2016, \arXiv{1603.06991v1}.
\bibitem[Ma14]{Ma} \bibaut{A. Massarenti}, \textit{The automorphism group of $\overline{M}_{g,n}$}, Journal of the London Mathematical Society, 2014, Vol. 89, 131-150,.
\bibitem[MM16]{MM} \bibaut{A. Massarenti, M. Mella}, \textit{On the automorphisms of Hassett's moduli spaces}, Transactions of the American Mathematical Society, 2016, DOI: 10.1090/tran/6966.
\bibitem[MM14]{MM14} \bibaut{A. Massarenti, M. Mella}, \textit{On the automorphisms of moduli spaces of curves}, Automorphisms in Birational and Affine Geometry, Springer Proceedings in Mathematics \& Statistics, 79, 2014, 149-167.
\bibitem[Mo13]{Moo}\bibaut{H. Moon}, \textit{Log canonical models for $\overline{M}_{g,n}$}, Proc. Amer. Math. Soc, 141, 2013, no. 11, 3771-3785.
\bibitem[Wi36]{Wi}\bibaut{E. Witt}, \textit{Zyklische K\"orper und Algebren der Characteristik $p$ vom Grad $p^{n}$. Struktur diskret bewerteter perfekter K\"orper mit vollkommenem Restklassenk\"orper der Charakteristik $p^{n}$}, J. Reine Angew. Math, 176, 1936, 126-140.
\end{thebibliography}
\end{document}